\documentclass[12pt]{amsart}

\usepackage{txfonts}
\newcommand{\attach}{+}
\usepackage{lmodern}

\usepackage{fullpage,amssymb,amsmath}
\usepackage{mathrsfs}
\usepackage{color}
\usepackage{algorithmicx}
\usepackage{algpseudocode}
\usepackage{scalerel}
\usepackage[vcentermath]{youngtab}
\usepackage{enumitem}

\newtheorem{theorem}{Theorem}[section]
\newtheorem{corollary}[theorem]{Corollary} 
\newtheorem{lemma}[theorem]{Lemma}
\newtheorem{proposition}[theorem]{Proposition}

\theoremstyle{definition}
\newtheorem{definition}[theorem]{Definition}
\newtheorem{remark}[theorem]{Remark}
\newtheorem{example}[theorem]{Example}
\newtheorem{problem}[theorem]{Problem}

\newcommand{\kar} {{\rm char}}

\newcommand{\C}{{\mathbb C}}

\newcommand{\Z}{{\mathbb Z}}

\newcommand{\Mat}{\operatorname{Mat}}
\newcommand{\Rep}{\operatorname{Rep}}
\newcommand{\GL}{\operatorname{GL}}
\newcommand{\SL}{\operatorname{SL}}
\newcommand{\Sym}{\operatorname{Sym}}

\newcommand{\Vect}{\operatorname{Vect}}

\newcommand{\spa}{\operatorname{span}}
\newcommand{\llambda} {\lambda^\dag}
\newcommand{\Spec}{{\rm Spec}}
\newcommand{\N} {\mathcal{N}}

\title{Weyl's polarization theorem in positive characteristic}
\author{Harm Derksen and Visu Makam}
\thanks{This material is based upon work supported by the National Science Foundation under Grant No. DMS-1601229 and DMS-1638352}
\begin{document}

\begin{abstract}
Let $V$ be an $n$-dimensional algebraic representation over an algebraically closed field $K$ of a group $G$.  For $m > 0$, we study the invariant rings $K[V^{ m}]^G$ for the diagonal action of $G$ on $V^{m}$. In characteristic zero, a theorem of Weyl tells us that we can obtain all the invariants in $K[V^{ m}]^G$ by the process of polarization and restitution from $K[V^{ n}]^G$. In particular, this means that if $K[V^{ n}]^G$ is generated in degree $\leq d$, then so is $K[V^{ m}]^G$ no matter how large $m$ is.

There are several explicit counterexamples to Weyl's theorem in positive characteristic.
However, when $G$ is a (connected) reductive affine group scheme over $\Z$ and $V^*$ is a good $G$-module, we show that Weyl's theorem holds in sufficiently large characteristic. As a special case, we consider the ring of invariants $R(n,m)$ for the left-right action of $\SL_n \times \SL_n$ on $m$-tuples of $n \times n$ matrices. In this case, we show that the invariants of degree $\leq n^6$ suffice to generate $R(n,m)$ if the characteristic is larger than $2n^6 + n^2$. 
\end{abstract}

\maketitle

\section{Introduction}
Let $K$ be an algebraically closed field. Suppose $V$ is a rational representation of a reductive group $G$. The ring of invariant polynomials $K[V]^G$ is a finitely generated graded subalgebra of the coordinate ring $K[V]$, see \cite{Haboush,Hilbert1,Hilbert2,Nagata}. A long standing theme in invariant theory is to extract a minimal set of generators -- apart from a few instances, this is perhaps too ambitious a problem. A more approachable problem is to find upper bounds on the degree of generators.

\begin{definition}
We define $\beta(K[V]^G)$ to be the smallest integer $D$ such that invariants of degree $\leq D$ form a generating set, i.e.,
$$
\beta(K[V]^G) := \min \{D\  |\  K[V]_{\leq D}^G \text{ is a generating set}\},
$$
where $K[V]^G_{\leq D}$ denotes the invariants of degree $\leq D$.
\end{definition}

A general bound for $\beta(K[V]^G)$ is given in \cite{Derksen1}. In this paper, we will be concerned with the growth of $\beta(K[V^m]^G)$ as $m$ gets large, where $V^m$ denotes the direct sum of $m$ copies of the representation $V$. It is easy to see that $\beta(K[V^a]^G) \leq \beta(K[V^b]^G)$ if $a \leq b$, and so for fixed $G$ and $V$, the sequence $\beta(K[V^m]^G)$ is increasing. In characteristic $0$, it is a remarkable result due to Weyl (see \cite{Weyl,KP}) that this sequence is actually bounded! 

\begin{theorem} [Weyl's polarization theorem -- weak form] \label{theo:weyl}
Assume $\kar (K) = 0$, and let $\dim V = n$. Then for all $m$, we have $\beta(K[V^m]^G) \leq \beta(K[V^n]^G)$. 
\end{theorem}

Weyl's result is actually a little stronger than the version we state above, which we will now discuss. Interpreting $V^m$ as $V \otimes K^m$ illuminates a $\GL_m$ action on $V^m$. Since this $\GL_m$ action commutes with the $G$ action, the invariant ring $K[V^m]^G = K[V \otimes K^m]^G$ inherits an action of $\GL_m$. For $a \leq b$, we have the inclusion $K[V^a]^G \subseteq K[V^b]^G$. Suppose $S$ is a generating set for $K[V^a]^G$. Starting with $S$, we can construct some obvious invariants in $K[V^b]^G$. For example, take any $f \in S$ and $g \in \GL_b$, then $g \cdot f \in K[V^b]^G$. In the same spirit, we can consider the smallest $\GL_b$-stable subspace  containing $S$, i.e., $\left<S\right>_{\GL_b} \subseteq K[V^b]^G$ (see Section~\ref{Sec:pol} for a more detailed definition). Consider the subalgebra of $K[V^b]^G$ generated by $\left<S\right>_{\GL_b}$. This subalgebra may not be all of $K[V^b]^G$, and there may be some other `genuinely new' invariants in $K[V^b]^G$. Weyl's polarization theorem says that there are no genuinely new invariants if we take $a$ at least as big as $n$.


\begin{theorem} [Weyl's polarization theorem -- strong form]
Assume $\kar (K) = 0$, and let $\dim V = n$. Let $S \subseteq K[V^n]^G$ be a generating set for $K[V^n]^G$. Then for all $m \geq n$, the set $\left<S \right>_{\GL_m}$ is a generating set for $K[V^m]^G$.
\end{theorem}

It is easy to see that the weak form of Weyl's theorem stated before is a consequence of the strong form stated above. In positive characteristic, one does not have to look far to get counterexamples. Suppose $\kar (K) = p > 0$. Let $C_p$ denote the cyclic group of order $p$, and consider the action of $C_p$ on $V = K^2$ where the generator of $C_p$ acts by the matrix $\begin{pmatrix} 1 & 1 \\ 0 & 1 \end{pmatrix}.$ Weyl's theorem fails in this case, see \cite{Richman}. Other examples of failure for finite groups can be found in \cite{Wehlau}. We note that finite groups are reductive in arbitrary characteristic. Knop showed in \cite{Knop} that the strong form of Weyl's theorem holds for invariant rings of finite groups if the characteristic 
is large enough. 

In this paper, we will restrict ourselves to connected reductive groups. Even in this restricted setting, Weyl's theorem still fails. For example, in characteristic $2$, it fails for the natural action of $G = {\rm SO}(V)$ on $V$ where $V$ is a four dimensional vector space, see \cite{DKZ}.
An analogue of Weyl's theorem in positive characteristic  was proved for {\em separating} invariants in \cite{DKW}.

\subsection{Matrix invariants and semi-invariants}
Let $\Mat_{p,q}$ denote the set of $p \times q$ matrices. Consider the group $G = \GL_n$ acting on $V = \Mat_{n,n}^m$ by simultaneous conjugation, i.e.,
$$
g \cdot(X_1,\dots,X_m) = (gX_1g^{-1},\dots,gX_mg^{-1}).
$$

We set $S(n,m) = K[V]^G$, the ring of invariants for this action. The ring $S(n,m)$ is often referred to as the ring of matrix invariants.Procesi showed that traces of monomials (in the $X_i$'s) generate $S(n,m)$ in characteristic $0$, see \cite{Procesi}. In \cite{Donkin,Donkin-2}, Donkin extended this result to all characteristics, by replacing traces with the coefficients of the characteristic polynomial instead. 

A bound on the degree of generators in characteristic $0$ followed from the work of Razmyslov, see \cite{Razmyslov}.

\begin{theorem} [Procesi--Razmyslov] \label{Pro-Raz}
Assume $\kar (K) = 0$. Then we have $\beta(S(n,m)) \leq n^2$.
\end{theorem}

One can observe that this bound is independent of $m$, as predicted by Weyl's theorem. It was pointed out to us by Domokos that the proof of the above result in \cite{Formanek} goes through once characteristic is larger than $n^2 + 1$. In particular, this means that the statement of Weyl's theorem holds for matrix invariants if we assume a modest lower bound on characteristic! However, the techniques used for this are very specific to matrix invariants, and it is not clear if they can be generalized. For example, even in the closely related example of matrix semi-invariants discussed below, such a result was not known prior to this paper. 

Consider the left-right action of $G = \SL_n \times \SL_n$ on $V = \Mat_{n,n}^m$, i.e., for $(A,B) \in \SL_n \times \SL_n$ and $(X_1,\dots,X_m) \in \Mat_{n,n}^m$, we have 
$$
(A,B) \cdot (X_1,\dots,X_m) = (AX_1B^{-1},\dots,AX_mB^{-1}).
$$

We set $R(n,m) = K[V]^G$, the invariant ring in this case. The ring $R(n,m)$ is often referred to as the ring of matrix semi-invariants. In recent times, connections to computational complexity has generated a lot of interest in matrix semi-invariants, see \cite{DM,GGOW,IQS,Mul}.

A determinantal description for the generators follows from results on semi-invariants of quivers, see \cite[Corollary 3]{DW},  \cite{DZ} and \cite{SVd}. A polynomial bound on the degree of generators was given in \cite{DM,DM-arbchar}. 

\begin{theorem} [\cite{DM,DM-arbchar}]  \label{msi-bounds}
Let $n \geq 2$. We have $\beta(R(n,m)) \leq mn^3(n-1)$. If $\kar (K) = 0$, then we have $\beta(R(n,m)) \leq n^6$.
\end{theorem}

The bound stated in \cite{DM,DM-arbchar} for $\beta(R(n,m))$ was $mn^4$, but these slightly stronger bounds are evident in the proof of \cite[Theorem~1.2]{DM}. The bound in characteristic $0$ is once again a consequence of Theorem~\ref{theo:weyl}. We prove that the statement of Weyl's theorem holds for matrix semi-invariants with only a modest lower bound on the characteristic. 

\begin{theorem} \label{msi.main}
Suppose $\kar (K) = p > 2n^6 + n^2$. Then the statement of Weyl's polarization theorem holds for the left-right action of $\SL_n \times \SL_n$ on $\Mat_{n,n}$.
In particular, for all $m \in \Z_{> 0}$ we have 
$$
\beta(R(n,m)) \leq \beta(R(n,n^2)) \leq n^6.
$$
\end{theorem}

Our techniques give similar results for matrix invariants as well, but the lower bound on characteristic we obtain is weaker than the already known $n^2 + 1$.

\begin{remark}
In small characteristic (i.e., $p \leq n$), the statement of Weyl's theorem is false for matrix invariants, see \cite{DKZ,Domokos}. By a standard reduction, the same phenomenon holds for matrix semi-invariants as well. However, it remains an open problem to understand whether the statement of Weyl's theorem holds for matrix invariants for $n < p \leq n^2 + 1$ and matrix semi-invariants for $n < p \leq 2n^6 + n^2$.
\end{remark}

We can further decrease the lower bound on characteristic if all we want is a bound that doesn't depend on $m$. However, the degree bound will become a bit worse. For example, the techniques in this paper can be used to show the following:

\begin{proposition} \label{unnecessary bound}
Suppose $\kar(K) = p > n^6$, then for all $m \in \Z_{>0}$;
$$
\beta(R(n,m)) \leq \beta(R(n,n^3)) \leq n^7.
$$ 
\end{proposition}

However, with these techniques, one cannot decrease the lower bound on characteristic to $O(n^{6 - \epsilon})$ for any $\epsilon > 0$.

\subsection{Main results}
We need some technical definitions for which we follow \cite{Seshadri}. An affine group scheme $G$ over $\Spec\  \Z$ (or simply $\Z$) is said to be reductive if $G \rightarrow \Spec\ \Z$ is smooth, and the geometric fibers are connected reductive algebraic groups (in the usual sense). Let $G$ be a reductive group scheme over $\Z$, and let $V$ be a free $\Z$-module of finite rank $n$ with a linear action of $G$. We will call $V$ a free $(G-\Z)$-module of rank $n$. We will denote the ring of invariants by $\Z[V]^G = \Sym(V^*)^G$. 

For any algebraically closed field $K$, the $K$-points $G_K$ form a connected reductive group over $K$, and the $K$-points of $V$, i.e., $V_K = V \otimes_\Z K$ is an $n$-dimensional representation of $G_K$. We will write $K[V] = K[V_K]$ and $K[V]^G = K[V_K]^{G_K}$ for simplicity. Note that $K[V]^G$ is not necessarily the same as the base change $\Z[V]^G \otimes_\Z K$.

\begin{definition} \label{delt}
Let $S = \bigoplus_{i \geq 0} S_i$ be a graded $R$-algebra. Then let $S_{\{d\}}$ denote the $R$-subalgebra generated by $\cup_{i \leq d} S_i$. Further, let $\delta_R(S)$ denote the smallest $d$ such that $S$ is a finite extension over $S_{\{d\}}$.
\end{definition}

The following theorem requires the notion of a good modules, which we recall in Section~\ref{Sec:good1}. A reductive group scheme over $\Z$ is called split if there is a (fiberwise) maximal torus defined over $\Z$.

\begin{theorem} \label{main}
Let $G$ be a split reductive group scheme over $\Z$, and let $V$ be a free $(G-\Z)$ module of rank $n$. Suppose $V^*$ is a good $G$-module. Then, the following statements hold.

\begin{enumerate}
\item The number $Q = \max\{2,\frac{3}{8}n (\delta_\Z(\Z[V^n]^G))^2\}$ is finite.
\item Suppose $K$ is an algebraically closed field such that $\kar(K) > 2 Q(n+1) + n$. Then the statement of Weyl's polarization theorem holds for the action of $G_K$ on $V_K$, i.e.,

\begin{enumerate}
\item  if $S$ is a set of generators for $K[V^n]^G$, then $\left<S\right>_{\GL_m}$ is a set of generators for $K[V^m]^G$ for all $m \geq n$;
\item we have $\beta(K[V^m]^G) \leq \beta(K[V^n]^G)$ for all $m \geq 1.$
\end{enumerate}
\end{enumerate}
\end{theorem}

For the first part of the theorem, we will need some results of Seshadri from \cite{Seshadri}. The bulk of the paper will go towards proving the second part of the above theorem. The approach is a delicate interplay between combinatorics, representation theory and commutative algebra.

\subsection{Organization}
In Section~\ref{Sec:prelim}, we recall some necessary preliminaries. We give a short proof of Weyl's polarization theorem in characteristic $0$ in Section~\ref{Sec:Weyl0}.
 We study polarization in Section~\ref{Sec:pol}. Then, in Sections~\ref{Sec:good1} and \ref{Sec:good2}, we discuss good filtrations. In Section~\ref{Sec:technical}, we discuss the technical details needed, and prove Theorem~\ref{msi.main}. Finally, in Section~\ref{Sec:final proofs}, we bring together all the results to prove the main result, i.e., Theorem~\ref{main}.

\section{Preliminaries} \label{Sec:prelim}
\subsection{Partitions}
A partition $\lambda = (\lambda_1,\lambda_2,\dots)$ is a (weakly) decreasing sequence of non-negative numbers, such that only finitely many $\lambda_i$ are non zero. We often omit writing the trailing zeros. We say $\lambda$ is a partition of $n$ if $\sum_i \lambda_i = n$, and we write $\lambda \vdash n$. Associated to any partition is its Young diagram. For example, if $\lambda = (4,3,1,1)$, then its Young diagram is 

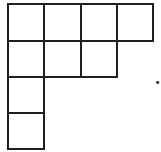
\begin{equation*} \label{egtab}
\yng(4,3,1,1). 
\end{equation*}

We will not distinguish between a partition and its Young diagram. For a partition $\lambda$, we define its size $|\lambda| := \sum_i \lambda_i = $ number of boxes in the Young diagram, and its length $l(\lambda)$ = length of the first column in its Young diagram. For the above example, we have $|\lambda| = 9$ and $l(\lambda) = 4$. We define $\lambda^{\dag}$ to be the conjugate of the partition $\lambda$.


\begin{definition} [Horizontal concatenation]
Given two partitions $\lambda$ and $\mu$, we define their horizontal concatenation $\lambda + \mu = (\lambda_1 + \mu_1, \lambda_2 + \mu_2, \dots)$. Note that $\lambda + \mu$ is a partition. 
\end{definition}

\begin{example}
We have $\yng(3,3,3,2)\  +  \ \yng(5,2) = \yng(8,5,3,2)$
\end{example}

\subsection{Schur functors}
For any commutative ring $R$, any $R$-module $E$ and any partition $\lambda$, one can construct a Schur module $S_{\lambda}(E)$ (denoted $E^{\lambda}$ in  \cite[Section~8.1]{Fulton}). Let $E^{\times \lambda}$ denote the direct product of $|\lambda|$ copies of $E$ labelled by boxes in the Young diagram of $\lambda$. The Schur module $S_{\lambda}(E)$ is defined as the universal target for $R$-module maps from $E^{\times \lambda}$ that are multilinear, alternating along columns, and satisfying some exchange relations. We do not recall the exchange relations, but refer instead to \cite[Section~8.1]{Fulton} for details. 

Let $K$ be an algebraically closed field. For any partition $\lambda$, the aforementioned construction gives a polynomial functor $S_{\lambda}: \Vect \rightarrow \Vect$, where $\Vect$ represents the category of finite dimensional vector spaces (over $K$). We call $S_{\lambda}$ the Schur functor associated to $\lambda$. We have $S_{(n)} = \Sym^n$, the $n^{th}$ symmetric power, whereas $S_{(n)^\dag} = S_{1^n} = \bigwedge^n$, the $n^{th}$ alternating power. Note that $S_{\lambda}$ is denoted by $L_{\lambda^\dag}$ in \cite{ABW,Wey}.

We require the following result that is well known to experts.

\begin{proposition} \label{schur:surj}
Let $\lambda$ and $\mu$ be two partitions. Then, there is a $GL(V)$-equivariant surjection $S_{\lambda}(V) \otimes S_{\mu}(V) \twoheadrightarrow S_{\lambda \attach \mu} (V)$ 
\end{proposition}

We will discuss a stronger statement, i.e., Corollary~\ref{L-R} later using the theory of good filtrations. Here, we indicate a combinatorial proof of the above proposition for the reader who is more familiar with Young tableaux.

\begin{proof} [Proof of Proposition~\ref{schur:surj}] 
One way to construct the partition $\lambda + \mu$ is to take all the columns of (the Young diagrams) of $\lambda$ and $\mu$ and rearrange them in decreasing order. This gives a map $V^{\times \lambda} \times V^{\times \mu} \twoheadrightarrow V^{\times (\lambda \attach \mu)} \twoheadrightarrow S_{\lambda + \mu}(V)$. We leave it to the reader to check that this map factors to give a surjective map $S_{\lambda}(V) \otimes S_{\mu}(V) \twoheadrightarrow S_{\lambda \attach \mu} (V)$ as required.
\end{proof}

\subsection{Polynomial representations of $\GL_m$ of degree $n$} 
We will only need Corollary~\ref{cat:ss} from this section, but a general reference for the definitions and results in this section is \cite{Totaro}. We first note that $S_{\lambda}(V)$ is a representation of $\GL(V)$. It is an irreducible representation in characteristic $0$, but not necessarily in positive characteristic. We denote by $\Rep^{\rm pol}(\GL_m)_d$, the category of polynomial representations of $\GL_m$ of degree $d$. This category is a highest weight category, and the costandard objects are precisely the Schur modules $S_{\lambda}(K^m)$ for $|\lambda| = d$. Totaro was able to give upper bounds on the homological dimension of this category, and compute it precisely under mild assumptions, see \cite{Totaro}.

\begin{theorem}[Totaro] \label{Totaro}
Let $\kar(K) = p$, and let $\alpha_p(d)$ denote the sum of the digits in the base $p$ expansion of $d$. The homological dimension of $\Rep^{\rm pol}(\GL_m)_d$ is $ \leq 2(d - \alpha_p(d))$. Further, we have equality if $m \geq d$.
\end{theorem}

\begin{corollary} \label{cat:ss}
Assume $p>d$. Then $\Rep^{\rm pol}(\GL_m)_d$ is semi-simple. Further, the Schur modules $S_{\lambda}(K^m)$ for $\lambda \vdash d$ are irreducible representations of $\GL_m$. 
\end{corollary}

\begin{proof}
The semisimplicity of $\Rep^{\rm pol}(\GL_m)_d$ follows from the aforementioned Totaro's theorem on homological dimension. The costandard objects in any semisimple highest weight category are irreducible. Hence the Schur modules $S_{\lambda}(K^m)$'s with $\lambda \vdash d$ are irreducible.
\end{proof}







\section{Weyl's theorem in characteristic zero} \label{Sec:Weyl0}
We give a short proof of Weyl's theorem in characteristic zero based on the representation theory of the general linear group. Let $G$ be a group defined over a field $K$ of characteristic zero, and let $V$ be an $n$-dimensional representation. For any $m$, we identify $V^m$ with $V \otimes K^m$, where the action of $G$ on $K^m$ is trivial. Now, by Cauchy's formula, we can write 

$$
K[V \otimes K^m] = \Sym(V^*\otimes K^m) = \bigoplus_{\lambda} S_\lambda(V^*) \otimes S_\lambda(K^m).
$$

The direct sum in the above is over partitions of all sizes. The crucial observation we need is that if $l(\lambda) > n$, then $S_\lambda(V^*) = 0$. So, only partitions that have $l(\lambda) \leq n$ give non-trivial summands. Combining this with taking $G$-invariants, we get
$$
K[V \otimes K^m]^G  = \bigoplus_{l(\lambda) \leq n} S_\lambda(V^*)^G \otimes S_\lambda(K^m).
$$

One can interpret this as the isotypic decomposition of $K[V \otimes K^m]^G$ with respect to the action of $\GL_m$. The various irreducibles appearing in this decomposition are of the form $S_\lambda(K^m)$, and $S_\lambda(V^*)^G$ is the multiplicity space.

For $m \geq n$, we have an inclusion $K[V \otimes K^n]^G \hookrightarrow K[V \otimes K^m]^G$. Let $S$ be a set of generators for $K[V \otimes K^n]^G$. Let us denote by $R$ the subalgebra of $K[V \otimes K^m]^G$ that is generated by $\left<S\right>_{\GL_m}$. We want to show that $R$ is all of $K[V \otimes K^m]^G$. Since $R$ is $\GL_m$-stable and contains $K[V \otimes K^n]^G$, it suffices to show that the smallest $\GL_m$-stable subspace containing $K[V \otimes K^n]^G$ is all of $K[V \otimes K^m]^G$. Using the isotypic decomposition from above, it suffices to show that $\left<S_\lambda(V^*)^G \otimes S_\lambda(K^n)\right>_{\GL_m} = S_\lambda(V^*)^G \otimes S_\lambda(K^m)$ for all $\lambda$ such that $l(\lambda) \leq n$. It is easy to see that it suffices to prove that $\left<S_\lambda(K^n) \right>_{\GL_m} = S_\lambda(K^m)$.

Now, observe that $S_\lambda(K^m)$ is an irreducible representation of $\GL_m$ and so has no proper $\GL_m$-stable subspaces. Since $S_\lambda(K^n)$ is non-empty for $l(\lambda) \leq n$, we have $\left< S_\lambda(K^n) \right>_{\GL_m} = S_\lambda(K^m)$.


\section{Polarization} \label{Sec:pol}
Let $E$ be a $\GL(W)$ representation. For any subset $S \subseteq E$, recall that we define $\left< S \right>_{\GL(W)}$ to be the smallest $\GL(W)$ stable subspace containing $S$. This is often referred to as polarization. In more concrete terms $\left< S \right>_{\GL(W)}$ consists of elements $e \in E$ that can be written as a sum $e = \sum_i g_i s_i$ with $s_i \in S$ and $g_i \in \GL(W)$. Let us note here that the definition of $\left<S\right>_{\GL(W)}$ depends on the ambient $\GL(W)$ representation $E$. For our discussion, it will almost always be obvious what the ambient representation is.

Understanding the following special case is the most crucial part of this paper.

\begin{problem} \label{prob:pol.schur}
For an inclusion of vector spaces $V \subseteq W$, we have $S_\lambda(V) \subseteq S_\lambda(W)$. When is $\left<S_{\lambda}(V) \right>_{\GL(W)} = S_{\lambda}(W)$?
\end{problem}

In characteristic $0$, this is always true as long as $S_{\lambda}(V)$ is non-empty, because the module $S_{\lambda}(W)$ is an irreducible $GL(W)$-module. This was a crucial part in the proof of Weyl's theorem in characteristic zero in the preceding section. In positive characteristic, this is often not the case as the following example shows:

\begin{example}
Suppose $\kar (K) = 2$, and let $\lambda = (2)$, so $S_{\lambda} = \Sym^2$. Consider $K^1 \hookrightarrow K^2$, and let $x,y$ be a basis for $K^2$ with $x$ being a basis for $K^1$. Then we have $\Sym^2(K^1) = \spa (x^2)$, where as $\Sym^2(K^2) = \spa(x^2,y^2,xy)$. It is easy to see that $\left< \Sym^2(K^1) \right>_{\GL_2} = \spa(x^2,y^2)$ which is a proper subset of $\Sym^2(K^2)$.
\end{example}
 
\begin{remark}
If $\dim V \geq |\lambda|$, we will always have $\left<S_{\lambda}(V) \right>_{\GL(W)} = S_{\lambda}(W)$. This is a simple consequence of the description of the Schur module in terms of semistandard Young tableaux. We need a much stronger statement to be of any use for our purposes.
\end{remark}

\begin{proposition} \label{prop:tensor:pol}
Suppose $V = V_1 \oplus V_2 \subseteq W$. Further, suppose we have $\left<S_{\mu}(V_1) \right>_{\GL(W)} = S_{\mu}(W)$ and $\left< S_{\nu}(V_2) \right>_{\GL(W)} = S_{\nu}(W)$. Then $\left<S_{\mu}(V_1) \otimes S_{\nu} (V_2) \right>_{\GL(W)} = S_{\mu}(W) \otimes S_{\nu}(W)$.
\end{proposition}

\begin{proof}
Consider $E \in S_{\mu}(W)$ and $F \in S_{\nu}(W)$. We will show that $E \otimes F \in \left<S_{\mu}(V_1) \otimes S_{\nu} (V_2) \right>_{\GL(W)}$. Since we have $\left<S_{\mu}(V_1) \right>_{\GL(W)} = S_{\mu}(W)$, we can write $E = \sum_{i} g_ie_i$ for some $g_i \in \GL(W)$ and $e_i \in S_{\mu}(V_1)$. Similarly, we can write $F = \sum_{j} h_jf_j$ for some $h_j \in \GL(W)$ and $f_j \in S_{\nu}(V_2)$.

Decompose $W = W_1 \oplus W_2$ with $V_i \subseteq W_i$. Let $(w_1,\dots,w_k)$ be a basis for $W_1$ and $(w'_1,\dots,w'_{l})$ be a basis for $W_2$. Let $(w_1,\dots,w_k,w'_1,\dots,w'_l)$ be an ordered basis for $W$. In this ordered basis, we have a block decomposition $g_i = [ A_i \ |\ B_i]$, where $A_i$ represents the first $k$ columns, and $B_i$ the last $l$ columns. Observe that since $e_i \in S_{\mu}(V_1)$, the action of $g_i$ on $e_i$ only depends on $A_i$. Similarly, write $h_j = [P_j\ |\ Q_j]$, and the action of $h_j$ on $f_j$ only depends on $Q_j$. Hence, if we define $\sigma_{ij} = [A_i \ |\ Q_j]$, we have 
$$
E \otimes F = \sum_{i,j} \sigma_{ij} \cdot (e_i \otimes f_j).
$$

There is a small issue that $\sigma_{ij}$ may not be invertible, but this is easy to circumvent. For some non-zero constant $c_{ij}$, we have $c_{ij}I + \sigma_{ij}$ is invertible, where $I$ denotes the identity transformation. Then we can write 

$$
E \otimes F = \sum_{i,j} (c_{ij} I + \sigma_{ij}) \cdot (e_i \otimes f_j) - (c_{ij}I) \cdot (e_i \otimes f_j).
$$

Since $c_{ij} I + \sigma_{ij}$ as well as $c_{ij}I$ are elements of $\GL(W)$, we have that
$$E \otimes F \in \left<S_{\mu}(V_1) \otimes S_{\nu} (V_2) \right>_{\GL(W)}.$$ 
The proposition follows since elements of the form $E \otimes F$ span $S_\mu (W) \otimes S_{\nu}(W)$.
\end{proof}

\begin{theorem} \label{theo:pol:main}
Let $\lambda = \mu \attach \nu$, and let $V_1 \oplus V_2 = V \hookrightarrow W$. Further, suppose we have $\left<S_{\mu}(V_1) \right>_{\GL(W)} = S_{\mu}(W)$ and $\left< S_{\nu}(V_2) \right>_{\GL(W)} = S_{\nu}(W)$. Then we have $\left<S_{\lambda}(V)\right>_{\GL(W)} = S_{\lambda}(W)$.
\end{theorem}

\begin{proof}
Consider the surjection $\pi: S_{\mu}(W) \otimes S_{\nu}(W) \twoheadrightarrow S_{\lambda}(W)$ from Proposition~\ref{schur:surj}. It suffices to show $\pi(E \otimes F) \in \left<S_{\lambda}(V)\right>_{\GL(W)}$ for $E \in S_{\mu}(W)$ and $F \in S_{\nu}(W)$. Indeed by the Proposition~\ref{prop:tensor:pol}, we have $E \otimes F = \sum_i g_i \cdot (e_i \otimes f_i)$ for $g_i \in \GL(W), e_i \in S_{\mu}(V_1)$, and $f_i \in S_{\nu}(V_2)$. Thus we have $\pi(E \otimes F) = \sum_i g_i \cdot(\pi(e_i \otimes f_i)) \in \left< S_{\lambda}(V) \right>_{\GL(W)}$.
\end{proof}

\begin{lemma}
Let $\lambda \vdash d$ such that $l(\lambda) \leq n$. Fix $k \geq 2$. Then we can write $\lambda = \mu_1 \attach \mu_2 \attach \dots \attach \mu_s$ for some positive integer $s$ and non-empty partitions $\mu_i$ for $i = 1,\dots,s$ such that $ n(k-1) < |\mu_i| \leq kn$ for all $i < s$ and $|\mu_s| \leq kn$. Further, we have $l(\mu_i) \leq n$ for all $i$.
\end{lemma}

\begin{proof}
Suppose $|\lambda| \leq kn$, then there is nothing to do. So, let us assume $|\lambda| > kn$. The lengths of the columns in $\lambda$ are given by the conjugate partition $\lambda^{\dag} = (\llambda_1,\llambda_2,\dots)$. Let $t$ be the smallest integer such that $\sum_{i = 1}^t \llambda_i > kn$. Then let $\mu_1$ be the first $(t-1)$ columns of $\lambda$, so that $\lambda = \mu_1 \attach \nu$, where $\nu$ is a partition. We have $kn \geq \sum_{i = 1}^{t-1} \llambda_i$ by minimality of $t$, and we have $\sum_{i = 1}^{t-1} \llambda_i > kn - n$, since $\llambda_t \leq n$ by hypothesis. Hence we have $n(k-1) < |\mu_1| \leq kn$.
 Now, proceed by induction on $\nu$.
\end{proof}

\begin{example}
Suppose $n = 4$ and $k = 3$ and $\lambda = (8,8,7,4)$, then the decomposition in the above lemma is best visualized by the following picture.
$$
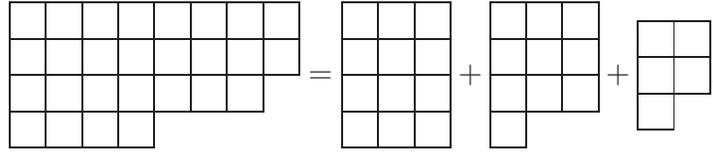
\yng(8,8,7,4) = \yng(3,3,3,3) + \yng(3,3,3,1) + \yng(2,2,1).
$$

\end{example}

\begin{corollary}
Let $\lambda \vdash d$ such that $l(\lambda) \leq n$, and suppose $\kar (K) > kn$ with $k \geq 2$. Then for $V \hookrightarrow W$, with $\dim V  \geq n\lceil \frac{d}{n(k-1)}\rceil$,
we have $\left<S_{\lambda}(V) \right>_{\GL(W)} = S_{\lambda}(W)$
\end{corollary}

\begin{proof}
Write $\lambda = \mu_1 \attach \mu_2 \attach \dots \attach \mu_s$ be the decomposition from the previous lemma. If $s \geq \lceil\frac{d}{n(k-1)} \rceil + 1$, then since $|\mu_i| > n(k-1)$ for $i < s$, and $|\mu_s| > 0$, we have $|\lambda| = \sum_i |\mu_i| > (s-1)n(k-1) = \lceil\frac{d}{n(k-1)} \rceil n(k-1) \geq d$, which is a contradiction. Hence we have $s \leq \lceil\frac{d}{n(k-1)} \rceil$, and consequently, we have $sn \leq n\lceil \frac{d}{n(k-1)}\rceil \leq \dim V$. This allows us to choose subspaces $V_1,\dots,V_s \subseteq W$ such that $\dim V_i = n$ and $V_1 \oplus V_2 \oplus \dots \oplus V_s \subseteq V$. 

First observe that $S_{\mu_i}(V_i)$ is non-zero as $l(\mu_i) \leq \dim V_i = n$. Next, we see from Corollary~\ref{cat:ss} that $S_{\mu_i}(W)$ is an irreducible $\GL(W)$ representation as $|\mu_i| \leq kn < \kar(K)$. Hence, we have $\left<S_{\mu_i}(V_i)\right>_{\GL(W)} = S_{\mu_i}(W)$. The result follows by a repeated application of Theorem~\ref{theo:pol:main}.
\end{proof}

\begin{corollary} \label{cor:bdd.pol}
Suppose $L:\Vect \rightarrow \Vect$ is a functor such that it has a filtration (of functors) whose subquotients are of the form $S_{\lambda}$ with $\lambda \vdash d$ and $l(\lambda) \leq n$. Suppose $\kar (K)  > kn$ with $k \geq 2$, and let $V \in \Vect$ such that $\dim V \geq n\lceil \frac{d}{n(k-1)}\rceil$. Then for $V \hookrightarrow W$, we have $\left< L(V) \right>_{\GL(W)} = L(W)$
\end{corollary}

\section{Good filtrations and the Littlewood--Richardson rule} \label{Sec:good1}
The theory of good filtrations is very powerful in positive characteristic. A comprehensive introduction to this theory can be found in \cite{Donkin2} (see also \cite{Don, Don2, Don3, Mat}). We also refer the reader to \cite{Domokos2,Zubkov} for an exposition with a view of using them for invariant rings coming from quivers including matrix invariants and semi-invariants.
 
 Let $G$ be a connected reductive algebraic group over an algebraically closed field $K$. Let $B$ be a choice of Borel subgroup of $G$ and let $T \subset B$ be a maximal torus of $G$. Let $\Lambda^+$ denote the set of dominant integral weights. For each $\lambda \in \Lambda^+$, one can associate a one-dimensional representation of $B$. The corresponding induced G-module is called a dual Weyl module, and denoted $\nabla(\lambda)$. Note that for $\GL_n$ and $\SL_n$, Schur modules are dual Weyl modules. There is a partial order $\prec$ on $\bigwedge^+$ defined by $\lambda \prec \mu$ if $\mu - \lambda$ is a non-negative sum of roots.

\begin{definition}
A $G$-module $V$ is called a good $G$-module if it has a filtration $0 \subseteq V_0 \subseteq V_1 \subseteq~\dots$ such that $\bigcup\limits_i V_i = V$ and each subquotient $V_i/V_{i-1}$ is a dual Weyl module. Such a filtration is called a good filtration. 
\end{definition} 

The dual Weyl modules occuring as subquotients (including multiplicities) are independent of the choice of filtration. 

\begin{remark} \label{defined over Z}
For a split reductive group defined over $\Z$, Weyl modules and dual Weyl modules are defined over $\Z$, see \cite{Jantzen} or \cite{Kulkarni}. More precisely, for $\lambda \in \Lambda^+$, there is a free $(G-\Z)$ module $\nabla_\Z(\lambda)$ such that $\nabla_\Z(\lambda) \otimes_\Z K$ is the dual Weyl module $\nabla(\lambda)$ for $G_K$ for any algebraically closed field $K$. So, we call a free $(G-\Z)$-module a good $G$ module if it has a filtration by the dual Weyl modules $\nabla_\Z(\lambda)$'s.  Moreover, the characters of dual Weyl modules are given by the Weyl character formula and in particular independent of the characteristic.
\end{remark}

The following lemma is straightforward.

\begin{lemma} \label{semi.good}
Suppose a $G$-module $V$ has a filtration $0 \subseteq V_0 \subseteq V_1 \subseteq~\dots$ such that $\bigcup\limits_i V_i = V$ and each subquotient $V_i/V_{i-1}$ is a good $G$-module, then $V$ is a good $G$-module.
\end{lemma}

Let us recall some well known properties of good $G$-modules. They can be found in the standard references mentioned above.

\begin{lemma} \label{good-prop}
Let $V$ and $W$ be good $G$-modules.
\begin{enumerate}
\item If $V \subseteq W$, then $W/V$ is a good $G$-module.
\item $V \otimes W$ is a good $G$-module.
\item $\dim(V^G)$ is the multiplicity of the trivial module in any good filtration for $G$.
\end{enumerate}
\end{lemma}

The following result is \cite[Proposition~3.2.6]{Donkin2}.

\begin{lemma}
Suppose $V$ is a good $G$-module. Suppose it has a good filtration $0 = V_0 \subseteq V_1 \subseteq \dots \subseteq V_n = V$ with $V_i/V_{i-1} = \nabla(\lambda_i)$. Let $\pi$ be a permutation of $\{1,2,\dots,n\}$ such that whenever $\lambda_{\pi(i)} \succ \lambda_{\pi(j)}$, we have $\pi(i) > \pi(j)$. Then there is a good filtration $0 = V'_0 \subseteq V'_1 \subseteq \dots \subseteq V'_n = V$ such that $V_i/V_{i-1} = \nabla(\lambda_{\pi(i)})$.
\end{lemma}

The following result already evident in the proof of the universal form of the Littlewood--Richardson rule (see \cite{Boffi}).  However, we will provide a sketch of the proof. Let us note that the dominance order on partitions agrees with the partial order $\prec$ for $G = \GL(V)$.

\begin{corollary} \label{L-R}
Suppose $\lambda,\mu$ are two partitions, and $V$ a vector space over an algebraically closed field $K$. Then we have a surjection $\zeta: S_\lambda(V) \otimes S_\mu(V) \twoheadrightarrow S_{\lambda \attach \mu} (V)$ such that $\ker(\zeta)$ has a filtration whose subquotients are Schur modules of the form $S_\nu(V)$ with $\nu \prec \lambda \attach \mu$.
\end{corollary}

\begin{proof} Since $S_\lambda(V)$ and $S_\mu(V)$ are good $\GL(V)$ modules, so is $S_\lambda(V) \otimes S_\mu(V)$ by Lemma~\ref{good-prop}. To understand the multiplicities of dual Weyl modules in any good $\GL(V)$-filtration for $S_\lambda(V) \otimes S_\mu(V)$, it suffices to write its character as a sum of characters of dual Weyl modules. This is a computation that is independent of characteristic as the dual Weyl modules have the same formal character in any characteristic, see Remark~\ref{defined over Z}.

In characteristic zero, the celebrated Littlewood--Richardson rule describes how $S_\lambda(V) \otimes S_\mu(V)$ decomposes as a sum of Schur modules. Hence, in any characteristic, the Littlewood--Richardson rule describes the subquotients in any good filtration of $S_\lambda(V) \otimes S_\mu(V)$. The Schur module $S_{\lambda + \mu}(V)$ occurs with multiplicity one, and all others are of the form $S_\nu(V)$ with $\nu \prec \lambda + \mu$. 

Using the above lemma, we can get a good filtration $0 = V_0 \subseteq \dots \subseteq V_k = S_\lambda(V) \otimes S_\mu(V)$ such that $V_k/V_{k-1} = S_{\lambda + \mu}(V)$. Interpreting this as a map $\zeta: V_k \twoheadrightarrow S_{\lambda + \mu}(V)$ whose kernel is $V_{k-1}$, we get the required conclusion.

\end{proof}

\section{Good filtrations for invariant rings} \label{Sec:good2}

For this section, let us assume $G$ is a connected reductive group over an algebraically closed field $K$ whose characteristic is $p > 0$, and $V$ is an $n$-dimensional good $G$-module. The following lemma is \cite[Lemma~2]{Zubkov}.

\begin{lemma} 
The module $\bigwedge^i(V)$ is a good $G$-module for $i < p$.
\end{lemma}

We can use the above lemma to prove the more general statement.

\begin{lemma}
If $p > n$, then $S_\lambda(V)$ is a good $G$-module for all partitions $\lambda$. 
\end{lemma}

\begin{proof}
The minimal elements in the dominance order on partitions are the partitions of the form $1^t$. We note that $S_{1^t}(V) = \bigwedge^t(V)$. From the above lemma, we see that all of these are good $G$-modules. We proceed by induction. Let $\lambda$ be a partition such that $S_\mu(V)$ is a good $G$-module for all $\nu$ smaller than $\lambda$ in the dominance order. If $l(\lambda) > n$, then $S_\lambda(V) = 0$, so we can assume $l(\lambda) \leq n < p$. Since $l(\lambda) < p$, we can write $\lambda = \mu \attach 1^t$ where $\mu$ is a partition and $t \leq l(\lambda) < p$. 

By the inductive hypothesis, $S_\mu(V)$ is a good $G$-module. Further, we have already observed that $S_{1^t}(V) = \bigwedge^t(V)$ is also a good $G$-module. Hence, by Lemma~\ref{good-prop},  $M = S_{\mu}(V) \otimes S_{1^t}(V)$ is a good $G$-module. By Corollary~\ref{L-R}, we have a surjection $\zeta: M \twoheadrightarrow S_{\mu \attach 1^t}(V) = S_\lambda(V)$ such that $\ker(\zeta)$ has a filtration by Schur modules $S_\nu(V)$ satisfying $\nu \prec \lambda$. By induction, all such $S_\nu(V)$'s are good $G$-modules. This means that $\ker(\zeta)$ has a filtration by good $G$-modules, and so by Lemma~\ref{semi.good}, we conclude that $\ker(\zeta)$ is a good $G$-module. By Lemma~\ref{good-prop}, we conclude that $S_\lambda(V) = M/\ker(\zeta)$ is also a good $G$-module. 
\end{proof}

\begin{corollary} \label{large.char.invring.good}
If $p > n$, then $\Sym(V)$ is a good $G$-module.
\end{corollary}

\begin{lemma} \label{lemma.after.cor}
If $p > n$, the module $\Sym(V \otimes W)$ is a good $G$-module for any finite dimensional $W$ ($G$ acts trivially on $W$).
\end{lemma}

\begin{proof}
We have $\Sym(V \otimes W) = \Sym(V)^{\otimes \dim W}$. So, it is a good module by Lemma~\ref{good-prop}.
\end{proof}

The following result first appeared in \cite{DRS}, but can also be found in \cite{ABW}.

\begin{theorem} [Doubilet-Rota-Stein]
$\Sym^d(V \otimes W)$ has a natural filtration whose associated graded module is $$\bigoplus_{\lambda \vdash d} S_{\lambda}(V) \otimes S_{\lambda}(W).$$
\end{theorem}

\begin{corollary}
$\Sym^d(V \otimes W)$ has a natural filtration whose associated graded module is $$\bigoplus_{\lambda \vdash d, l(\lambda) \leq n} S_{\lambda}(V) \otimes S_{\lambda}(W).$$
\end{corollary}
\begin{proof}
This follows from the above theorem, since $S_{\lambda}(V) = 0$ if $l(\lambda) > n$.
\end{proof}

\begin{lemma}
Suppose $p > n$. Then $\Sym^d(V \otimes W)^{G}$ has a natural filtration whose associated graded module is $$\bigoplus_{\lambda \vdash d, l(\lambda) \leq n} S_{\lambda}(V)^{G} \otimes S_{\lambda}(W).$$
\end{lemma}

\begin{proof}
Let $0 = F_0 \subseteq F_1 \subseteq \dots \subseteq F_m = \Sym^d(V \otimes W)$ denote the $G \times \GL(W)$-filtration from the above corollary. We know that $F_i/F_{i-1}$ is of the form $S_{\lambda}(V) \otimes S_{\lambda}(W)$, so $(F_i/F_{i-1})^G = S_{\lambda}(V)^G \otimes S_{\lambda}(W)$

Observe that $\Sym^d(V \otimes W)^G$ has a filtration 
$$0 = F_0^G \subseteq F_1^G \subseteq \dots \subseteq F_m^G = \Sym^d(V \otimes W)^G.$$ 
The associated graded module of this filtration is $\oplus_i F_i^G/F_{i-1}^G$.
Hence, if we show that $F_i^G/F_{i-1}^G = (F_i/F_{i-1})^G$, we would be done. It is easy to see that we have natural injective maps $\eta_i: F_i^G/F_{i-1}^G \hookrightarrow (F_i/F_{i-1})^G$ for each $i$. So, it suffices to show that the maps $\eta_i$ are isomorphisms. We show this by counting the dimension of $\Sym(V \otimes W)^G$ in two ways. 

First, observe that $\sum_i \dim(F_i^G/F_{i-1}^G) = \dim\Sym^d(V \otimes W)^G$ by a simple telescoping argument. On the other hand, $\dim(F_i/F_{i-1})^G$ is the multiplicity of the trivial module in any good filtration for $F_i/F_{i-1}$ by Lemma~\ref{good-prop}. Now, consider a good filtration for each quotient $F_i/F_{i-1}$, and then lift them to get a good filtration of $\Sym^d(V\otimes W)$. Thus the multiplicity of the trivial module in such a good filtration is $\sum_i \dim(F_i/F_{i-1})^G$ which is again equal to $\dim \Sym^d(V \otimes W)^G$ by Lemma~\ref{good-prop}. Thus we have $\sum_i \dim(F_i^G/F_{i-1}^G) = \sum_i \dim(F_i/F_{i-1})^G$. But now since each $\eta_i$ is an injection, it follows that they must all be isomorphisms.
\end{proof}

\begin{corollary} \label{bounded partitions}
Suppose $p > n$, then $\Sym^d(V \otimes W)^G$ has a natural filtration whose associated subquotients are all of the form $S_{\lambda}(W)$ with $\lambda \vdash d$ and $l(\lambda) \leq n$.
\end{corollary}

\begin{corollary} \label{inv.ring.pol.low.degrees}
Assume $p > kn$ for some $k \geq 2$. Let $W$ be a vector space with $\dim W \geq n\lceil \frac{d}{n(k-1)} \rceil$. Then, for any inclusion of vector spaces $W \hookrightarrow W'$, we have 
$$
 \left< \Sym^d(V \otimes W)^G \right>_{\GL(W')} = \Sym^d(V \otimes W')^G.
 $$
\end{corollary}

\begin{proof}
First observe that $p > kn > n$. Consider the polynomial functor $L$ defined by $L(U) = \Sym(V \otimes U)^G$. $L$ has a filtration by $S_{\lambda}$ with $\lambda \vdash d$ and $l(\lambda) \leq n$ by the previous corollary. Hence, by Corollary~\ref{cor:bdd.pol}, we have $\left< L(W) \right>_{\GL(W')} = L(W')$.
\end{proof}

\section{Technical details} \label{Sec:technical}
We discuss a few elementary results before proceeding to the main technical result.
\subsection{Decomposable elements}
For a graded ring $R = \bigoplus_{d=0}^\infty R_d$, we define the notion of decomposable and indecomposable elements.

\begin{definition}
A homogeneous element $f \in R_d$ of degree $d$ is called decomposable if it can be written as $f = \sum_{i \in I} g_i h_i$, where $g_i,h_i$ are homogeneous elements of degree $< d$. If a homogeneous element is not decomposable, we call it indecomposable.
\end{definition}

\begin{corollary}
The set $R_{\leq N} = \bigoplus_{i =1}^N R_i$ is a set of generators for $R$ if and only if for all $d > N$, every element of $R_d$ is decomposable.
\end{corollary}

For the rest of the section, let $V$ is a representation of a group $G$ over some algebraically closed field $K$.


\begin{lemma} \label{dec.inv.pol}
The set of decomposable invariants in $K[V^m]^G_d = \Sym^d(V^* \otimes K^m)^G$ is $\GL_m$ stable.
\end{lemma}

\begin{proof}
Suppose $f \in K[V^m]^G_d$ is decomposable, and $\sigma \in \GL_m$. Then we can write $f = \sum_i u_iv_i$, where $u_i,v_i$ are homogeneous invariants of degree $< d$. Hence we have $\sigma(f) = \sum_i \sigma(u_i) \sigma(v_i)$. Hence $\sigma(f)$ is also decomposable since $\sigma(u_i)$ and $\sigma(v_i)$ are also homogeneous invariants of degree $< d$.
\end{proof}

\begin{lemma} \label{gen.set.to.all}
Assume $a \leq b$ and let $S$ be a set of generators for $K[V^a]^G$. If $\left<K[V^a]^G_d \right>_{\GL_b} = K[V^b]^G_d$ for all $d \leq \beta(K[V^b]^G)$, then $\left<S\right>_{\GL_b}$ is a generating set for $K[V^b]^G$.
\end{lemma}

\begin{proof}
Clearly, it suffices to show that every indecomposable invariant in $K[V^b]^G$ can be generated by $\left<S\right>_{\GL_b}$. Take an indecomposable invariant $f \in K[V^b]^G$. It has degree $d \leq \beta(K[V^b]^G)$. Thus $f \in\left<K[V^a]^G_d \right>_{\GL_b}$ by hypothesis. Hence, we have $f  = \sum_i g_i\cdot f_i$ with $f_i \in K[V^a]^G_d$ and $g_i \in \GL_b$. But since $S$ is a generating set for $K[V^a]^G$, we can write each $f_i = p_i(s_{i_1},\dots,s_{i_{r_i}})$ for some polynomial $p_i$ in $r_i$ variables, and $s_{i_j} \in S$. Thus we have 
$$
f = \sum_i g_i\cdot f_i = \sum_i g_i \cdot p_i(s_{i_1},\dots,s_{i_{r_i}}) = \sum_i p_i(g_i\cdot s_{i_1},\dots,g_i\cdot s_{i_{r_i}})
$$

But this means that $f$ is generated by $\left<S\right>_{\GL_b}$.

\end{proof}

\subsection{Main technical result}
For this section, let $G$ be a reductive group over an algebraically closed field $K$ of characteristic $p$, and let $V$ be an $n$-dimensional representation such that $V^*$ is a good $G$-module. This section is devoted to proving the following result.

\begin{proposition} \label{Weyl.technical}
Suppose $Q \geq \frac{1}{2}$ is such that $\beta(K[V^m])^G \leq mQ$ for all $m \geq 1$ and $p > 2Q(n+1) + n$. If $S$ is a generating set for $K[V^n]^G$, then  $\left<S\right>_{\GL_m}$ is a generating set for $K[V^m]^G$ for all $m\geq n$. 
\end{proposition}
To prove this proposition, we need some other results first.

\begin{lemma}
Assume $\kar(K) > k n$ for some $k \geq 2$. If $m \geq n \lceil \frac{d}{n(k-1)} \rceil$, then we have $\left< K[V^m]^G_d \right>_{\GL_{m+1}} = K[V^{m+1}]^G_d$.
\end{lemma}

\begin{proof}
This follows from Corollary~\ref{inv.ring.pol.low.degrees} because we assume that $V^*$ is a good $G$-module.
\end{proof}

\begin{corollary} \label{mimic}
Assume the hypothesis of Proposition~\ref{Weyl.technical}. Suppose $m \geq n$. Then, we have $\left< K[V^m]^G_d \right>_{\GL_{m+1}} = K[V^{m+1}]^G_d$ for all $d \leq \beta(K[V^{m+1}]^G)$.
\end{corollary}

\begin{proof}
It follows from the hypothesis that  $d \leq \beta(K[V^{m+1}]^G) \leq  (m+1) Q $. Also, observe that by hypothesis, we have $p=\kar(K) > kn$, where $k = 2 Q(1 + 1/n) + 1 \geq 2$. So, in order to use the above lemma, we only need to show 
$$
m \geq  n  \left\lceil \frac{d}{n(k-1)}\right\rceil.
$$
We have two cases:

\begin{description}
\item [Case 1: $\frac{d}{n(k-1)} \leq 1$] In this case, we have $m \geq n = n \lceil \frac{d}{n(k-1)} \rceil$ by hypothesis. 

\item [Case 2: $\frac{d}{n(k-1)} > 1$] In this case, we have $\lceil \frac{d}{n(k-1)} \rceil < 2 \frac{d}{n(k-1)}$, hence it suffices to show 
$$
m \geq 2 n \frac{d}{n(k-1)} = \frac{2d}{k-1}.
$$
But this is the same as showing that $d \leq m (k -1) /2$. We know that $d \leq (m+1) Q$, so it suffices to show that $(m+1) Q \leq m (k-1)/2$. Rearranging, we need to show that $k \geq 2 Q(1 + 1/m) + 1$. But it is easy to see that $k =  2 Q(1 + 1/n) + 1 \geq 2 Q(1 + 1/m) + 1$ since $m \geq n$. 

\end{description}
\end{proof}

\begin{proof} [Proof of Proposition~\ref{Weyl.technical}]
This follows from Lemma~\ref{gen.set.to.all} and a repeated application of the above corollary.
\end{proof}

\begin{proof} [Proof of Theorem~\ref{msi.main}]
For $n =1$, the result is obvious. Assume $n \geq 2$, and now the result follows from Proposition~\ref{Weyl.technical} once we observe that the hypothesis is satisfied for $Q = n^3(n-1)$ by Theorem~\ref{msi-bounds}. One does need to be a little careful in applying Proposition~\ref{Weyl.technical} as $\dim(\Mat_{n,n}) = n^2$ and not $n$. 
\end{proof}

\begin{proof} [Proof of Proposition~\ref{unnecessary bound}]
One has to mimic the proof of Proposition~\ref{Weyl.technical}. First, note that we have to replace $n$ by $n^2$ because $\dim(\Mat_{n,n}) = n^2$. In Corollary~\ref{mimic}, one needs to adjust the assumption to $m \geq n^3$. Accordingly the two cases in the proof should be modified. Let us write $T = \frac{d}{n^2(k-1)}$. Then the two cases one should use are that either $\lceil T \rceil \leq n$ or $\lceil T \rceil \leq (1+\frac{1}{n}) T$. 
\end{proof}


\begin{remark}
Similar results can be formulated for various invariant rings and semi-invariant rings associated to quivers, by standard reductions to  be reduced to the case of matrix invariants and semi-invariants, see \cite{DM,DM2,DM-arbchar} for details.
\end{remark}

\section{Proof of main theorem} \label{Sec:final proofs}
The results in the previous section are perfectly good on their own. They are applicable to a number of cases in which we know an explicit $Q$ that satisfies the hypothesis of Proposition~\ref{Weyl.technical}. The importance of Theorem~\ref{main} is to provide a plethora of examples. The only catch though is that part $(1)$ of Theorem~\ref{main} is an existential result rather than an explicit one. It is an interesting problem to give any kind of explicit bound on the number $Q$. 


In characteristic zero, an approach to upper bounds for the degree of generators of invariant rings was proposed by Popov, and improved by the first author in \cite{Derksen1}. For a collection of polynomials $T$, we denote its zero locus by $\mathbb{V}(T)$.

\begin{definition}
Let $V$ be a representation of a reductive group $G$ over an algebraically closed field $K$. The null cone is given by
$$
\N(G,V) := \mathbb{V}\Big(\bigcup\limits_{d=1}^\infty K[V]^G_d\Big) \subseteq V.
$$
Further, define
$$
\gamma_K(G,V) := \min\left\{ D \ \Big|\ \mathbb{V}\Big(\bigcup\limits_{d=1}^D K[V]^G_d\Big) = \N(G,V) \right\}.
$$
\end{definition}

Now, let us study the null cone for the action of $G$ on several copies of $V$. It is easy to see that $\gamma_K(G,V^a) \leq \gamma_K(G,V^b)$ for $a \leq b$. However, we claim that the sequence stabilizes.

\begin{lemma}
Let $V$ be an $n$-dimensional representation of a reductive group $G$ over an algebraically closed field $K$. For all $a \geq 0$, we have 
$$
\gamma_K(G,V^a) \leq \gamma_K(G,V^n).
$$
\end{lemma}

\begin{proof}
Without loss of generality we can assume $a > n$. From the Hilbert--Mumford criterion for the null cone, it is easy to see that whether a tuple $(v_1,\dots,v_a) \in V^a$ is in the null cone or not depends only on the linear subspace $\spa (v_1,\dots,v_a) \subseteq V$. So, for any tuple $(v_1,\dots,v_a)$ not in the null cone, let $v_{i_1},\dots,v_{i_r}$ be a basis for $\spa (v_1,\dots,v_a)$. Then $(v_{i_1},\dots,v_{i_r})$ is not in the null cone for the action of $G$ on $V^r$, and observe that $r \leq n$. This gives an invariant $f \in K[V^r]^G$ such that $f(v_{i_1},\dots,v_{i_r}) \neq 0$. Now, define $\widetilde{f} \in K[V^a]^G$ by $\widetilde{f}(w_1,\dots,w_a) = f(w_{i_1},\dots,w_{i_r})$. Then $\widetilde{f}(v_1,\dots,v_a) = f(v_{i_1},\dots,v_{i_r}) \neq 0$. This means there is an invariant of degree $\leq \gamma_K(G,V^r) \leq \gamma_K(G,V^n)$ that doesn't vanish on $(v_1,\dots,v_a)$. This means that the null cone for $V^a$ is cut out by invariants of degree $\leq \gamma_K(G,V^n)$. This proves the lemma.
\end{proof}

If $V$ is a free $(G-\Z)$ module of rank $n$ for some reductive group scheme $G$ over $\Z$, then we write $\gamma_K(G,V) = \gamma_K(G_K,V_K)$ for simplicity. Recall the definition of $\delta_R(S)$ from Definition~\ref{delt}.

\begin{lemma}
Let $G$ be a reductive group scheme over $\Z$. Suppose $V$ is a free $(G-\Z)$ module of rank $n$. Then $\gamma_K(G,V) \leq  \delta_\Z(\Z[V]^G)$ for any algebraically closed field $K$.
\end{lemma}

\begin{proof}
This follows in two steps. The first is to see that $\gamma_K(G,V) = \delta_K(K[V]^G)$. This follows from \cite[Lemma~2.5.5, Remark~4.7.2]{DK}. The second is to check that $\delta_K(K[V]^G) \leq \delta_\Z(\Z[V]^G)$. First, we note that $\Z[V]^G$ is finitely generated, see \cite[Theorem~2]{Seshadri}, so in particular, we know that $\delta_\Z(\Z[V]^G) < \infty$.

By definition of $\delta_\Z(\Z[V]^G)$, we have that $\Z[V]^G$ is a finite extension over $\Z[G_1,\dots,G_m]$ where $G_i$ are homogeneous of degree $\leq \delta_\Z(\Z[V]^G)$. But finite extensions are preserved under base change, so $\Z[V]^G \otimes_\Z K$ is a finite extension over $\Z[G_1,\dots,G_m] \otimes_\Z K = K[G_1,\dots,G_s]$.  Hence, if we show that $K[V]^G$ is a finite extension of $\Z[V]^G \otimes_\Z K$, then it follows that $K[V]^G$ is a finite extension of $K[G_1,\dots,G_s]$. So, we can conclude that $\delta_K(K[V]^G) \leq \max\{\deg(G_i)\} \leq \delta_\Z(\Z[V]^G)$.

So, all it remains is to prove that $K[V]^G$ is a finite extension of $\Z[V]^G \otimes_\Z K$. As we noticed before, $\Z[V]^G$ is finitely generated, so $\Z[V]^G = \Z[F_1,\dots,F_r]$, and so $\Z[V]^G \otimes_\Z K = K[F_1,\dots,F_r]$. We see from \cite[Proposition~6]{Seshadri} that $\mathbb{V}(F_1,\dots,F_r) \subseteq V_K$ is the null cone for the action of $G_K$ on $V_K$. Again, by\cite[Lemma~2.5.5, Remark~4.7.2]{DK}, we see that $K[V]^G$ is a finite extension of $K[F_1,\dots,F_r]=\Z[V]^G \otimes_\Z K$ as required.
\end{proof}

Combining the above two lemmas, we get the following result.

\begin{corollary} \label{comparetoZ}
 Let $G$ be a reductive group scheme over $\Z$. Suppose $V$ is a free $(G-\Z)$ module of rank $n$. Then we have $\gamma_K(G,V^a) \leq \gamma_\Z(\Z[V^n]^G)$ for all algebraically closed fields $K$ and for all $a$.
\end{corollary}

\begin{proof}
From the previous two lemmas, we get $\gamma_K(G,V^a) \leq \gamma_K(G,V^n) \leq \gamma_\Z(\Z[V^n]^G)$. 
\end{proof}

We now discuss how the bounds for invariants defining the null cone translate into bounds for the degrees of generators. Let us state the main result in \cite{Derksen1} in a slightly different way.

\begin{theorem} \label{Derksen.bound}
Assume $K$ is an algebraically closed field of characteristic $0$. Let $V$ be an $n$-dimensional representation of a reductive group $G$ over $K$. Then we have
$$
\beta(K[V]^G) \leq \max\{2, {\textstyle \frac{3}{8}} n \gamma_K(G,V)^2\}.
$$
\end{theorem}

There are two reasons that we need characteristic zero in the above statement. The first is that invariant rings for reductive groups are Cohen--Macaulay in characteristic zero, see \cite{HR}. The second is Kempf's result that the Hilbert series is a proper rational function, see \cite{Kempf}. In the situations that we are interested in, these required ingredients are true in positive characteristic as well. The Cohen--Macaulay condition was addressed by Hashimoto in \cite{Hashimoto}.

\begin{theorem} [\cite{Hashimoto}]
Suppose $G$ is a reductive group over an algebraically closed field $K$. Suppose $V$ is a representation of $G$ such that $K[V] = \Sym(V^*)$ is a good $G$-module. Then $K[V]^G$ is strongly F-regular, and in particular Cohen--Macaulay.
\end{theorem}

To get Kempf's result on the Hilbert series to arbitrary characteristic, we use a comparison to the characteristic zero case. In order to compare across characteristics, we will need the representation and the reductive group to be defined over $\Z$. So, we will work under the hypothesis of the Theorem~\ref{main} for the following result.

\begin{proposition}
Let $G$ be a split reductive group scheme over $\Z$, and $V$ a free $(G - \Z)$-module of rank $n$ such that $V^*$ is a good $G$-module. Let $K$ be an algebraically closed field such that $K[V]$ is a good $G_K$-module (for e.g. if $\kar(K) > n$ by Corollary~\ref{large.char.invring.good}). Then the Hilbert series for $K[V]^G$ is the same as the Hilbert series for $\C[V]^G$. In particular, it is a proper rational function.
\end{proposition}

\begin{proof}
This fact that the Hilbert series of $K[V]^G$ is the same as the Hilbert series for $\C[V]^G$ has been observed before in the context of matrix invariants and matrix semi-invariants in  \cite{Donkin,Domokos,DM-arbchar}. The same proof works, and we sketch it.

First note that the group $G_\C$ is reductive, and hence linearly reductive. In particular, all $G_\C$ modules have good filtrations. Since $K[V]_d$ (resp. $\C[V]_d$) has a good filtration, to get $\dim(K[V]^G)$ (resp. $\dim\C[V]^G_d$), one has to write the character of $K[V]_d$ (resp. $\C[V]_d$) as a sum of characters of dual Weyl modules and read off the coefficient of the trivial character (see Lemma~\ref{good-prop}). The characters of the dual Weyl modules are independent of characteristic, so the two computations are identical. 
\end{proof}
Thus the degree bound for invariants in Theorem~\ref{Derksen.bound} continues to hold in positive characteristic if we add some hypothesis so as to use the above results. So, we get the following:

\begin{corollary}
Let $G$ be a split reductive group scheme over $\Z$, and $V$ a free $(G - \Z)$-module of rank $n$ such that $V^*$ is a good $G$-module. Let $K$ be an algebraically closed field such that $K[V]$ is a good $G_K$-module. Then we have
$$
\beta(K[V]^G) \leq \max\{2, {\textstyle \frac{3}{8}} n \gamma_K(G,V)^2\}.
$$
\end{corollary}

\begin{proof} [Proof of Theorem~\ref{main}]
To prove the first part, it suffices to show that $\delta_\Z(\Z[V^n]^G)$ is finite. This follows from the fact that $\Z[V^n]^G$  is finitely generated, see \cite[Theorem~2]{Seshadri}.

Let us now turn towards the second part. Let $K$ be an algebraically closed field with $\kar(K) > 2Q(n+1) + n$. Observe that that since $\kar(K) > n$ and $V^*$ is a good $G$-module, we have that  $K[V^m]$ is a good $G$-module for all $m$ by Lemma~\ref{lemma.after.cor}. So the above proposition applies and so for all $m$, we have 
$$
\beta(K[V^m]^G) \leq  \max\{2,{\textstyle \frac{3}{8}} mn \gamma_K(G,V^m)^2\} \leq \max\{2, m ({\textstyle \frac{3}{8}} n \delta_\Z(\Z[V^n]^G)^2)\} \leq m Q.
$$
The first inequality follows from the above proposition since $\dim(V^m) = mn$. The second inequality follows from Corollary~\ref{comparetoZ}, and the last inequality follows from the definition of $Q$. Since $\beta(K[V^m]^G) \leq mQ$, we can apply Proposition~\ref{Weyl.technical} and the second part of the theorem follows.
\end{proof}

To end with, we wish to emphasize an important future direction of research, namely, to produce a strong upper bound for $\delta_\Z(\Z[V^n])^G$. At the moment, we do not have any kind of explicit bound!

\subsection*{Acknowledgements}
We would like to thank Matyas Domokos, David Wehlau, Jerzy Weyman and Jakub Witaszek for helpful discussions.

\end{document}